\documentclass[12pt,oneside]{amsart}
\usepackage{amsmath,amssymb,latexsym,soul,cite,mathrsfs}
\usepackage{tikz}
\usepackage{color,graphicx}
\usepackage[colorlinks=true,urlcolor=blue,
citecolor=blue,linkcolor=blue,linktocpage,pdfpagelabels,
bookmarksnumbered,bookmarksopen]{hyperref}
\usepackage[english]{babel}
\usepackage{lineno}

\usepackage[left=2.7cm,right=2.7cm,top=3cm,bottom=3cm]{geometry}

\usepackage[hyperpageref]{backref}

\numberwithin{equation}{section}
\newtheorem{theorem}{Theorem}[section]
\newtheorem{lemma}[theorem]{Lemma}
\newtheorem{corollary}[theorem]{Corollary}
\newtheorem{proposition}[theorem]{Proposition}

\newtheorem{remark}[theorem]{Remark}

\newcommand{\R}{{\mathbb R}}

\renewcommand{\epsilon}{\varepsilon}

\renewcommand{\rightarrow}{\to}

\newcommand{\Real}{\mathbb{R}}

\title[]{Normalized solutions for  INLS equation with  critical Hardy-Sobolev type nonlinearities}  

\author[M. Cardoso]{M. Cardoso}

\author[J.F. de Oliveira]{J.  F.  de Oliveira}\thanks{Second author was supported by CNPq grant  number 309491/2021-5.}
\author[O.H. Miyagaki]{O. H. Miyagaki}\thanks{Third author was supported by Grant 2019/24901-3- S\~ao Paulo Research Foundation (FAPESP) and Grant 303256/2022-2 - CNPq/Brazil. }

\address[M. Cardoso]{\newline\indent Department of Mathematics
\newline\indent 
Federal University of Piau\'{\i}
\newline\indent
 64049-550 Teresina, PI, Brazil}
\email{\href{mailto:mykael@ufpi.edu.br}{mykael@ufpi.edu.br }}

\address[J.F. de Oliveira]{\newline\indent Department of Mathematics
\newline\indent 
Federal University of Piau\'{\i}
\newline\indent
 64049-550 Teresina, PI, Brazil}
\email{\href{mailto:jfoliveira@ufpi.edu.br}{jfoliveira@ufpi.edu.br}}

\address[O.H. Miyagaki]{\newline\indent Department of Mathematics
\newline\indent 
Federal University of S\~ao Carlos,
\newline\indent
13565-905 S\~ao Carlos, SP, Brazil}
\email{\href{mailto:ohmiyagaki@gmail.com}{ohmiyagaki@gmail.com}}


\begin{document}
\begin{abstract}  We are interested in finding  prescribed $L^2$-norm   solutions
to  inhomogeneous  
nonlinear Schr\"{o}dinger  (INLS) equations.  For $N\ge 3$ we treat  the  equation with  combined Hardy-Sobolev power-type nonlinearities
$$
 -\Delta u+\lambda u=\mu|x|^{-b}|u|^{q-2}u+|x|^{-d}|u|^{2^*_{d}-2}u \;\;\mbox{in}\;\; \mathbb{R}^N,\, N\ge 3
$$
where $\lambda\in\mathbb{R}$, $\mu>0$,  $0<b,d<2$, $2+(4-2b)/N<q<2+(4-2b)/(N-2)$ and $2^*_{d}= 2(N-d)/(N-2)$  is the Hardy-Sobolev critical exponent, while  for $N=2$ we investigate the equation with critical exponential growth
\begin{equation}\nonumber
\begin{aligned}
 &-\Delta u+\lambda u=|x|^{-b}f(u) \;\;\mbox{in}\;\; \mathbb{R}^2
\end{aligned}
\end{equation} 
where the  nonlinearity $f(s)$ behaves like $\exp(s^2)$ as $s\to\infty$.  We  extend the  existence results due to Alves-Ji-Miyagaki (Calc. Var.  61, 2022)  from $b =d= 0$ to the case $0 < b,d < 2$.
\end{abstract}
\maketitle



\section{Introduction and main results}
Let us consider  the inhomogeneous nonlinear Schr\"odinger (INLS) equation with combined power-type nonlinearities
\begin{equation} \label{PVI}
i \psi_{t} + \Delta \psi +\mu|\psi|^{q-2}\psi+|\psi|^{p-2}\psi=0, \,\,\, x \in \mathbb{R}^N,  \,t>0.
\end{equation}
The study of equations like  \eqref{PVI}  started with  the fundamental contribution by T. Tao, M. Visan and X. Zhang \cite{Tao} and there are a lot of recent advances, see for instance \cite{AJM,BMRV,Soave,SoaveJDE,jeanjean2,Li} and references quoted therein.

The solutions for the equation \eqref{PVI} of the form $\psi(x,t)=e^{i\lambda t}u(x)$, where  $\lambda\in\mathbb{R}$ and $u:\mathbb{R}^N\to\mathbb{R}$ is a time-independent function are called of \textit{standing waves}. So, to  find  standing waves, it is sufficient to solve  the stationary problem 
\begin{align}\label{Stading-eq}
-\Delta u+\lambda u=|u|^{q-2}u+|u|^{p-2}u \;\; \mbox{in}\;\;\mathbb{R}^N.
\end{align}
When it comes to equation \eqref{Stading-eq}, there is a large number of researches on  normalized solutions which means that
\begin{equation}\label{L2-presc}
\int_{\mathbb{R}^N}|u|^2dx=a^2,\;\; a>0.
\end{equation}
Normalized solutions to semilinear elliptic problems are investigated in different applied
models \cite{Cira, Pella} and the conditions on $N$, $p$ and $q$
 has strong repercussions on the existence of  solution $u$ and in aspects of it itself, see \cite{jeanjean1,BMRV,AJM,SoaveJDE,Soave}
and references therein.  

In order to be clear our aims, let   $N\ge 3$, $0\le b<2$  and $2<q\le 2+(4-2b)/(N-2)$.  According with \cite{CG, FGenoud} there exists unique positive radial solution $U\in H^{1}(\mathbb{R}^N)$  to the elliptic equation
\begin{align}\nonumber
-\Delta U+U=|x|^{-b}|U|^{q-2}U \;\; \mbox{in}\;\;\mathbb{R}^N\setminus\{0\}
\end{align}
having mass 
\begin{equation}\label{MassU}
m:=m(b,q)=\int_{\mathbb{R}^N}|U|^{2}dx>0.
\end{equation}  
For $\lambda>0$,  we can verify that the scaling 
\begin{equation}\label{scalingV}
v_{\lambda}(x)=\lambda^{\frac{2-b}{2(q-2)}}U(\lambda^{\frac{1}{2}}x),\;  x\in \mathbb{R}^N
\end{equation}
satisfies the equation 
\begin{equation}\label{pura}
-\Delta v+\lambda v=|x|^{-b}|v|^{q-2}v \;\; \mbox{in}\;\;\mathbb{R}^N\setminus\{0\}
\end{equation}
and
\begin{equation}
\int_{\mathbb{R}^N}|v_{\lambda}|^{2}dx=\lambda^{\frac{2-b}{q-2}-\frac{N}{2}}m.
\end{equation}
Consequently,  for   $q\not =2+(4-2b)/N$ and $\rho>0$, there exists unique $\lambda_{\rho}>0$ such that $v_{\lambda_{\rho}}$ given by \eqref{scalingV}  solves  \eqref{pura}   and has mass
\begin{equation}\nonumber
\int_{\mathbb{R}^N}|v_{\lambda_{\rho}}|^{2}dx=\rho.
\end{equation}
In addition, for the limiting case  $q=2+(4-2b)/N$, there are infinitely many solutions (one for each $\lambda>0$) of \eqref{pura} of the form \eqref{scalingV} having mass $m$. If $b=0$, 
the value $\overline{q}=2+4/N$ is  the so-called  $L^{2}$-\textit{critical exponent} which plays a special role in the study of \eqref{Stading-eq}, see for instance  \cite{AJM,Pella, SoaveJDE} for deeper discussion.

On the other hand, we can consider the value of the best Hardy-Sobolev constant \cite{CKN} on a domain $\Omega\subset\mathbb{R}^N,\, N\ge 3$  containing $0$ in its interior
\begin{equation}\label{Mus}
\mu_{s}(\Omega)=\inf\Big\{\int_{\Omega}|\nabla u|^{2}dx\;:\; u\in H^{1}_{0}(\Omega)\;\;\mbox{and}\;\; \int_{\Omega}|x|^{-s}|u|^{2^{*}_{s}}dx=1\Big\}
\end{equation}
where $0<s<2$ and $2^{*}_s=2(N-s)/(N-2)$ is the Hardy-Sobolev critical exponent. It is well known that  (c.f \cite{GY}) $\mu_s=\mu_s(\Omega)$ does not depend of $\Omega$ and it is attained when $\Omega=\mathbb{R}^N$ by the functions
\begin{equation}
v_{\epsilon}(x)=(\epsilon(N-s)(N-2))^{\frac{N-2}{4(2-s)}}(\epsilon+|x|^{2-s})^{\frac{2-N}{2-s}}
\end{equation}
 for some $\epsilon>0$. Moreover, the functions $v_{\epsilon}$ are the only positive radial solutions of 
 \begin{equation}
 -\Delta u= |x|^{-s} |u|^{2^{*}_s-2}u \;\;\mbox{in}\;\; \mathbb{R}^{N}.
 \end{equation}
For a more in-depth discussion on this topic, we recommend \cite{GY, Peral, GK} and references therein. 
 
Motivated by the above discussion,  we are mainly interested in the existence of normalized solutions for the following  inhomogeneous  
nonlinear Schr\"{o}dinger (INLS) equation with  Hardy-Sobolev combined power-type nonlinearities
\begin{equation}\label{problemPP}
\begin{aligned}
 & -\Delta u+\lambda u=\mu|x|^{-b}|u|^{q-2}u+|x|^{-d}|u|^{2^{*}_d-2}u\;\;\mbox{in}\;\; \mathbb{R}^N,
\end{aligned}
\end{equation}
 where $N\geq 3$, $\mu>0$, $0\le b,d<2$ and $2+(4-2b)/N<q<2^{*}_{b}$. For $N=2$ we also  treat  the equation with critical  exponential  growth
 \begin{equation}\label{TM-eq}
 \begin{aligned}
 -\Delta u+\lambda u=|x|^{-b}f(u)\;\;\mbox{in}\;\; \mathbb{R}^2,
\end{aligned}
\end{equation} 
where  $0\le b<2$ and  the  nonlinearity $f(s)$ behaves like $\exp(s^2)$ as $s\to\infty$.

 Our first  result is about  Hardy-Sobolev combined power-type  nonlinearities and  reads below:
\begin{theorem}\label{thm-HS}
Let $N\ge 3$, $a>0$, $0<b,d<2$ and $q\in (2+(4-2b)/N, 2^{*}_{b})$. Then, there exists $\mu^*=\mu^{*}(a,b)>0$ such that  for all $\mu\geq\mu^*$ the problem \eqref{problemPP} has a solution $(u,\lambda)\in H^{1}(\mathbb{R}^N)\times \mathbb{R}$ with $u$ satisfying \eqref{L2-presc}.
\end{theorem}
For $N=2$, it is well known that  the critical growth is no longer determined by any function of the power-type, but rather by exponential growth as seen by the so-called Trudinger-Moser  inequality, see for instance \cite{Miya,DDOR, LamLuIbero} and references therein.  In this sense,  for $0\le b<2$, we  say  that a continuous function $f:\mathbb{R}\to \mathbb{R}$ has critical  exponential growth if
\begin{equation}\label{TM-growth}
  \lim_{|t|\to+\infty}\frac{|f(t)|}{e^{\alpha\big(1-\frac{b}{2}\big) t^2}}=0\;\;\mbox{if}\;\; \alpha>4\pi\;\;\mbox{and}\;\;  \lim_{|t|\to+\infty}\frac{|f(t)|}{e^{\alpha\big(1-\frac{b}{2}\big) t^2}}=+\infty\;\;\mbox{if}\;\;0<\alpha<4\pi.
  \end{equation}
We also consider the  following  assumptions on $f$:
\begin{enumerate}
\item [$(f_1)$] $\lim_{t\to 0}\frac{|f(t)|}{|t|^{\tau}}=0$, for some $\tau>3$
\item [$(f_2)$] There is $\theta>4$ such that 
\begin{equation}\nonumber
0<\theta F(t)\le tf(t), \;\;\mbox{for}\;\; t\not=0, \;\;\mbox{where}\;\; F(t)=\int_{0}^{t}f(s)ds.
\end{equation}
\item [$(f_3)$] There exist $p>4$ and $\mu>0$ such that 
\begin{align*}
\mathrm{sgn}(t)f(t)\ge \mu|t|^{p-1}, \;\;\mbox{for}\;\; t\not=0
\end{align*}
where $\mathrm{sgn}:\mathbb{R}\setminus\left\{0\right\}\to \left\{-1, 1\right\}$ is such that $\mathrm{sgn}(t)=-1$ if $t<0$ and $\mathrm{sgn}(t)=1$ for $t>0$.
\end{enumerate}
\begin{theorem} \label{thm-TM}Let $0<b<2$. Assume that $f$ satisfies  \eqref{TM-growth} and $(f_1)$-$(f_3)$. If $a\in (0,1)$, then there exists $\mu^{*}=\mu^{*}(a, b)>0$ such that for all $\mu\ge \mu^{*}$  the  problem \eqref{TM-eq} admits  solution $(u_a,\lambda_a)\in H^{1}(\mathbb{R}^{2})\times \mathbb{R}$,  with $u_{a}$  satisfying \eqref{L2-presc} and $\lambda_{a}>0$.
\end{theorem}
The paper is organized as follows. In Section~\ref{sec2} we present some preliminaries results.  Section~\ref{sec3} is devoted to the proof of
Theorem~\ref{thm-HS}.  The existence of normalized solutions for critical exponential growth stated in Theorem~\ref{thm-TM} is established in Section~\ref{sec4}.

\paragraph{\textbf{Notation:}} Throughout this paper, unless otherwise stated,  we use the following notations:
\begin{itemize}
\item $B_{r}(x)$ is an open ball centered at $x$ with radius $r>0$ and  $B_r= B_r(0)$.
\item $\|\;\;\|_{L^{p}_b}$ denotes the norm of the weighted Lebesgue space $L^{p}_{b}(\mathbb{R}^N)=L^{p}(\mathbb{R}^N, |x|^{-b})$, for $b\ge 0$ $p\in [1, \infty]$ and $\|\;\;\|_{L^{p}}=\|\;\;\|_{L^{p}_0}$.
\item $\|\;\;\|$ denotes the  usual norm of the Sobolev space $H^{1}(\mathbb{R}^N)$.
\item $C,c, C_1, c_1, C_2, c_2, \dots$ denote any positive constant whose value is not relevant.
\item $o_{n}(1)$ denotes a real sequence with  $o_{n}(1)\to 0$ as $n\to\infty$.
\item $\rightarrow$ and $\rightharpoonup$ denote strong and weak convergence.
\end{itemize}

\section{Preliminary results}
\label{sec2}
In this section, as an application of a  Gagliardo-Nirenberg  type inequality present in \cite{DF,CG,Farah},  we provide a compactness result that will be used below.

Let us introduce the weighted Lebesgue space $ L^{q}_b(\Real^N)$ as the set of all measurable functions $u$ such that  $$\|u\|_{L^{q}_b}=\Big(\int_{\mathbb R^N} |x|^{-b} |u(x)|^{q}\,dx\Big)^{\frac{1}{q}} <\infty.$$

We start  by recalling the  Gagliardo-Nirenberg type inequality.
\begin{theorem}[\cite{DF,CG,Farah}]\label{GNU}
Let $0<b<\min\{N,2\}$ and $2<q<2^*_b$, then for all $u\in H^1(\mathbb R^N)$
\begin{align}\label{GNopt}
\|u\|^{q}_{L^{q}_b}\leq K_{opt}\|\nabla u\|_{L^2}^{\frac{N(q-2)}{2}+b} \|u\|_{L^2}^{q-\frac{N(q-2)}{2}-b},
\end{align}
 where $$K_{opt}=\left(\frac{N(q-2)+2b}{2(q-b)-N(q-2)}\right)^{-\frac{N(q-2)-4+2b}{4}}\frac{2q}{(N(q-2)+2b)\|Q\|_{L^2}^{q-2}}$$
 and $Q$ is the unique positive radial solution to the elliptic equation, 
\begin{align}\label{elptcpc}
-\Delta Q+Q=|x|^{-b}|Q|^{q-2}Q.
\end{align}
\end{theorem}
\noindent As  by-product of the Theorem~\ref{GNU}  we have the continuous embedding 
\begin{equation}\label{sub-emb}
H^1(\mathbb R^N) \hookrightarrow  L^{q}_b(\Real^N)
\end{equation} in the strict case $2<q<2^*_{b}$.
We observe that, although Theorem~\ref{GNU}  does not include  the critical case  $q=2^*_{b}$,  the embedding \eqref{sub-emb} remains valid for $N\ge 3$  and $q=2^*_{b}$. In fact, we are able to prove the inequality
\begin{equation}\label{GNopt*}
\|u\|_{L^{2^{*}_b}_{b}}\le c\|\nabla u\|_{L^2},\;\; \mbox{for any}\;\; u\in H^{1}(\mathbb{R}^N).
\end{equation}
Indeed,  from Hardy (cf. \cite[Lemma~2.1]{Peral}),  Sobolev and H\"{o}lder inequalities we can write
\begin{align*}
\int_{\mathbb{R}^N}|x|^{-b}|u|^{2^{*}_{b}}dx&=\int_{\mathbb{R}^N}|x|^{-b}|u|^{b}|u|^{2^{*}_{b}-b}dx\\
&\le C_1\left(\int_{\mathbb{R}^N}|\nabla u|^2dx\right)^{\frac{b}{2}}\left(\int_{\mathbb{R}^N}|u|^{2^*}dx\right)^{\frac{2-b}{2}}\\
&\le C_1\left(\int_{\mathbb{R}^N}|\nabla u|^2dx\right)^{\frac{b}{2}}\left(\int_{\mathbb{R}^N}|\nabla u|^{2}dx\right)^{\frac{2-b}{2}\frac{2^*}{2}}\\
&= C_1\left(\int_{\mathbb{R}^N}|\nabla u|^2dx\right)^{\frac{N-b}{N-2}}.
\end{align*}
Thus,  \eqref{GNopt*} holds. Now, according with \eqref{GNopt*} we obtain a positive constant $\Lambda$ given by 
\begin{align}\label{critenergcnst}
	\Lambda=\inf_{0\neq u\in D^{1,2}(\Real^N)}\displaystyle\frac{\displaystyle\|\nabla u\|^2_{L^2}}{\|u\|^{2}_{L^{2^{*}_{b}}_{b}}}>0.
\end{align}
We recommend \cite{GY}, for further discussion on this subject.

Next, we present a significant compact embedding between the Sobolev space and the weighted  Lebesgue space  $L^q_{b}(\Real^N)$.
\begin{proposition}\label{WSC}
	Let $N\geq 3$, $0<b<2$, $2<q<2^*_b$ and $2<p<2^*$. Then, the Sobolev embedding  
	\begin{align}
	\dot H^1(\Real^N)\cap L^{p}(\Real^N)\hookrightarrow L^{q}_b(\Real^N)
	\end{align}
	is compact.
\end{proposition}
\begin{proof}
	Let $(u_n)$ be a bounded sequence in $ \dot{H}^1(\Real^N)\cap L^{p}(\Real^N)$. Then, there exists $u\in \dot H^1(\Real^N)\cap L^{p}(\Real^N)$ such that $u_n\rightharpoonup u$ in $ \dot{H}^1(\Real^N)\cap L^{p}(\Real^N)$ as $n\to \infty$. Defining $w_n=u_n-u$, we will show that 
	\begin{align}
	\int_{\mathbb{R}^N}|x|^{-b}|w_n|^{q}\,dx\to 0,\;\;\mbox{as}\;\; n\to\infty.
	\end{align}
First, from the weak convergence, $(w_n)$ is uniformly bounded in $\dot H^1(\Real^N)\cap L^{p}(\Real^N)$, and thus, using the Sobolev embedding, we get that 
	\begin{align}\label{lmtaunif}
	(w_n)\mbox{ is uniformly bounded in } L^r(\Real^N)\mbox{ for all }p< r <2^*.
	\end{align}
	Moreover, for all $R>0$ and $\alpha>N$, we have that
	\begin{align}\label{intR}
	\int_{\Real ^N\backslash B_R}|x|^{-\alpha}\,dx\leq \frac{C}{R^{\alpha-N}},
	\end{align} 
for some $C>0$ depending only on $\alpha$ and $N$.  Since  $q <2^{*}_b$, we get $2(N-b)-q(N-2)>0$. Thus,  for any $0<\varepsilon < 2(N-b)-q(N-2)$ we are able to choose $\gamma_1$ and $\gamma^{\prime}_1$ such that 
\begin{equation}\nonumber
\left\{\begin{aligned}
 &\frac{1}{\gamma_1'} = \frac{b}{N} - \varepsilon,\;\;\mbox{i.e}\;\;  \gamma^{\prime}_1b > N\\
 &p < q\gamma_1 < 2^*\\
& \frac{1}{\gamma_1} + \frac{1}{\gamma_1'} = 1.
\end{aligned}\right.
\end{equation}
Thus, by H\"older's inequality, we have
	\begin{align}\label{foradabola}
	\int_{\Real^{N}\backslash B_R}|x|^{-b}\left|w_n\right|^{q}\,dx\leq \left(\int_{\Real^N\backslash B_R}|x|^{-b\gamma_1'}\,dx\right)^{\frac{1}{\gamma_1'}}\left(\int_{\Real^N\backslash B_R}|w_n|^{q\gamma_1}\,dx\right)^{\frac{1}{\gamma_1}},
	\end{align}
	for all $R>0$. From \eqref{lmtaunif}, \eqref{intR} and \eqref{foradabola}, we can choose $R>0$ such that
	\begin{align}\label{inteps}
	\int_{\Real^N\backslash B_R}|x|^{-b}|w_n|^{q}\,dx<\epsilon/2.
	\end{align}
	Now, we are going to estimate the integral on the ball $B_R$. For $R>0$ chosen in \eqref{inteps}, consider $\eta\in C_0^{\infty}(\Real^N)$ a cut-off function such that 
	\begin{equation}
	\eta(x)=
	\begin{cases}
	1,\,\,\,\,\mbox{ if }x\in B_R,\\
	0,\,\,\,\,\mbox{ if }x\in \Real^N\backslash B_{2R}
	\end{cases}
	\end{equation}
	and $|\nabla \eta(x)|\leq cR^{-N}$ for all $x\in \Real^N$, for some constant $c>0$. Thus, $$\eta {w_n}_{\left|_{B_R}\right.}=w_n,\;\;\eta {w_n}_{\left|_{B_{3R}}\right.}\in H^1_0(B_{3R})\;\;\mbox{and}\;\; \eta w_n\rightharpoonup 0 \;\;\mbox{in} \;\;\dot H_0^1(B_{3R}).$$ By compactness of the Sobolev embedding $ H^1_0\left(B_{3R}\right)\hookrightarrow L^r\left(B_{3R}\right)$,  it follows that 
	\begin{align}\label{convstrong}
	\eta w_n\to 0\mbox{ in } L^r\left(B_{3R}\right)\mbox{ strongly for } 2<r<2^*.
	\end{align}
Again, since $q<2^{*}_b$, we obtain $2(N-b)-q(N-2)>0$. Then, we can choose $\gamma_2$ and $\gamma^{\prime}_2$ such that $\gamma_2'\geq 1$, $\gamma_2'b<N$, $2<q<\gamma_2<2^*$ and  $\frac{1}{\gamma_2}+\frac{1}{\gamma_2'}=1$. Hence, by H\"older's inequality we have
	\begin{align}
	\int_{B_{R}}|x|^{-b}\left|w_n\right|^{q}\,dx&\leq \left(\int_{B_R}|x|^{-b\gamma_2'}\,dx\right)^{\frac{1}{\gamma_2'}}\left(\int_{B_R}\left|w_n\right|^{q\gamma_2}\,dx\right)^{\frac{1}{\gamma_2}}\\
	&\leq \left(\int_{B_R}|x|^{-b\gamma_2'}\,dx\right)^{\frac{1}{\gamma_2'}}\left(\int_{B_{3R}}|\eta w_n|^{q\gamma_2}\,dx\right)^\frac{1}{\gamma_2},
	\end{align}
	and thus, together with \eqref{convstrong}, there exists $n_0$ such that for any $n\geq n_0$
	\begin{align}\label{in Ball}
	\int_{B_R}|x|^{-b}\left|w_n\right|^{q}\,dx<\epsilon/2.
	\end{align}
Thus, the result follows from  \eqref{inteps} and \eqref{in Ball}.
\end{proof}
\section{The minimax strategy}
\label{sec3}
\noindent Throughout this section we are assuming $N, a, b,d$ and $q$ under the assumptions of Theorem~\ref{thm-HS}. Let us denote
\begin{equation}\label{f-fuction}
    f(x,t)= \mu|x|^{- b} |t|^{q-2}t +|x|^{-d}|t|^{2_d^*-2}t,\;x\in\mathbb{R}^N\;\;\mbox{and}\;\; t\in\mathbb{R}
\end{equation}
and its primitive 
\begin{equation}\label{F-primitiva}
    F(x,t)=\frac{\mu}{q} |x|^{- b} |t|^{q}+\frac{1}{2_d^*} |x|^{- d} |t|^{2_d^*}.
\end{equation} 
In order to get our existence results  we will apply the minimax approach. In fact, the aim is to obtain  a nontrivial critical point $u\in H^{1}(\mathbb{R}^N)$ of the functional $J:  H^{1}(\mathbb{R}^N)\to \mathbb{R}$ be defined by
\begin{equation}
J(u)=\frac{1}{2}\int_{\mathbb{R}^{N}}|\nabla u|^{2}d x-\int_{\mathbb{R}^N}F(x,u)dx
\end{equation}
submitted to the constraint
\begin{equation}
S(a)=\left\{u\in H^{1}(\mathbb{R}^N)\; :\; \|u\|_{L^2}=a \right\}.
\end{equation}
We observe that \eqref{GNopt} and \eqref{GNopt*} yield the well-definition of $J$ and if $u\in H^{1}(\mathbb{R}^N)$ is a critical point of $J$ submitted to the constraint $S(a)$, then there exists a Lagrange multiplier $\lambda$ such that $(u,\lambda)$ satisfies \eqref{problemPP}. 

In order to follow the variational procedure in \cite{jeanjean1}, we will introduce  the space $H= H^{1}(\mathbb{R}^N)\times \mathbb{R}$ equipped with the scalar product $\langle\cdot
,\cdot \rangle_{H}=\langle\cdot
,\cdot \rangle_{H^{1}(\mathbb{R}^{N})}+\langle\cdot
,\cdot \rangle_{\mathbb{R}}$ and the application $\mathcal{H}: H\to  H^{1}(\mathbb{R}^{N})$ given by
\begin{equation}
\mathcal{H}(u,s)=e^{\frac{Ns}{2}}u(e^{s}x).
\end{equation}
Note that
\begin{equation}\label{grad-H}
   \int_{\mathbb{R}^N}|\nabla\mathcal{H}(u,s)|^{2}dx=e^{2s}\int_{\mathbb{R}^N}|\nabla u|^{2}dx,
\end{equation}
\begin{equation}\label{Lp-H}
    \int_{\mathbb{R}^N}|x|^{-b}|\mathcal{H}(u,s)|^{q}dx=e^{\frac{[N(q-2)+2b]s}{2}}\int_{\mathbb{R}^N}|x|^{-b}|u|^{q}d x.
\end{equation}
In particular, 
\begin{equation}\label{L2-H}
    \int_{\mathbb{R}^N}|\mathcal{H}(u,s)|^{2}dx=\int_{\mathbb{R}^N}|u|^{2}d x.
\end{equation}
Then, if  $\tilde{J}: H \to \mathbb{R}$ is defined by
\begin{equation}\label{Hfunctional}
\begin{aligned}
\tilde{J}(u,s)& = \frac{e^{2s}}{2}\int_{\mathbb{R}^{N}}|\nabla u|^{2} d x-\frac{1}{e^{Ns}}\int_{\mathbb{R}^N}F(e^{-s}x, e^{\frac{Ns}{2}}u(x))dx\\
&= \frac{e^{2s}}{2}\|\nabla u\|^{2}_{L^2}-\frac{\mu e^{\frac{[N(q-2)+2b]s}{2}}}{q}\|u\|^{q}_{L^{q}_b}-\frac{e^{\frac{[N(2^{*}_d-2)+2d]s}{2}}}{2^*_d}\|u\|^{2^{*}_d}_{L^{2^{*}_d}_d}
\end{aligned}
\end{equation}
we obtain
\begin{equation}\label{JJtil}
\begin{aligned}
\tilde{J}(u,s)=J(v), \;\;\mbox{for}\;\; v=\mathcal{H}(u,s).
\end{aligned}
\end{equation}
\begin{lemma}\label{vicentej}
Let $u\in S(a)$ be fixed arbitrarily. Then
\begin{itemize}
\item [$(i)$]  If $s\to -\infty$, then  $\|\nabla\mathcal{H}(u,s)\|_{L^2}\to 0$ and $J(\mathcal{H}(u,s))\to 0$.
\item [$(ii)$] If $s\to \infty$, then  $\|\nabla\mathcal{H}(u,s)\|_{L^2}\to \infty$ and $J(\mathcal{H}(u,s))\to -\infty$.
\end{itemize}
\end{lemma}
\begin{proof}
Note that $0<b<2$ and $2+(4-2b)/N<q<2^{*}_b$ yield $N(q-2)-(4-2b)>0$. Also, it is clear that $N(2^*_d-2)-(4-2d)>0$ for any $0<d<2$. In addition,  from \eqref{Hfunctional} and \eqref{JJtil} 
\begin{equation}\label{JonH}
\begin{aligned}
    J(\mathcal{H}(u,s))&=\frac{e^{2s}}{2}\Big[\|\nabla u\|^{2}_{L^2}-\frac{2\mu e^{\frac{[N(q-2)-(4-2b)]s}{2}}}{q}\|u\|^{q}_{L^{q}_{b}}-\frac{2 e^{\frac{[N(2^*_d-2)-(4-2d)]s}{2}}}{2_d^*}\|u\|^{2^{*}_{d}}_{L^{2^{*}_d}_{d}}\Big].
\end{aligned}    
\end{equation}
Then, $(i)$ and $(ii)$ follow directly from \eqref{grad-H} and \eqref{JonH}.
\end{proof}

\begin{lemma}  \label{PJ1} There exists $K=K(a,b, \mu)>0$ small enough such that
	$$
	0<\sup_{u\in A} J(u)<\inf_{u\in B} J(u)
	$$
with
$$
A=\big\{u\in S(a)\; :\; \|\nabla u\|^2_{L^2}\le K \big\},\quad B=\big\{u\in S(a)\; :\; \|\nabla u\|^2_{L^2}=2K \big\}.
$$
In addition, $K(a, b,\mu)\rightarrow 0$ as $\mu\rightarrow \infty.$
\end{lemma}
\begin{proof}
Fix $u,v\in S(a)$ with $u\in A$ and  $v\in B$. From  \eqref{GNopt} and \eqref{critenergcnst}
\begin{equation}\label{Fv-step1}
\int_{\mathbb{R}^N}F(x,v)\,dx \leq \frac{\mu}{q}K_{opt} a^{q-\frac{N(q-2)}{2}-b}(\|\nabla v\|^2_{L^2})^{\frac{N(q-2)+2b}{4}}+\frac{1}{2^*_d \Lambda^{\frac{2^*_d}{2}}}\left(\|\nabla v\|_{L^2}^2\right)^{\frac{2^*_d}{2}}.
\end{equation}
Since  $F(x,u)\geq 0$ for any $u\in H^{1}(\mathbb{R}^{N})$, we have
\begin{eqnarray*}
J(v)-J(u) &=&\frac{1}{2}\int_{\mathbb{R}^N}|\nabla v|^{2}\,dx-\frac{1}{2}\int_{\mathbb{R}^N}|\nabla u|^{2}\,dx-\int_{\mathbb{R}^N}F(x,v)\,dx+\int_{\mathbb{R}^N}F(x,u)\,dx\\
&\geq &\frac{1}{2}\int_{\mathbb{R}^N}|\nabla v|^{2}\,dx-\frac{1}{2}\int_{\mathbb{R}^N}|\nabla u|^{2}\,dx-\int_{\mathbb{R}^N}F(x,v)\,dx,
\end{eqnarray*}
and so,
\begin{align*}
    & J(v)-J(u) \ge \frac{K}{2}-\frac{\mu}{q}K_{opt} a^{q-\frac{N(q-2)}{2}-b}(2K)^{\frac{N(q-2)+2b}{4}}-\frac{1}{2^*_d \Lambda^{\frac{2^*_d}{2}}}(2K)^{\frac{2^*_d}{2}}\\
    &=K\Bigg[\frac{1}{2}-\frac{\mu}{q}K_{opt} a^{q-\frac{N(q-2)}{2}-b}2^{\frac{N(q-2)+2b}{4}}K^{\frac{N(q-2)+2b-4}{4}}-\frac{1}{2^*_d \Lambda^{\frac{2^*_d}{2}}}2^{\frac{2^*_d}{2}}K^{\frac{2^*_d-2}{2}}\Bigg]
\end{align*}
Thus, since $N(q-2)-(4-2b)>0$ we can choose $K>0$ small enough 
$$
J(v)-J(u)\ge \frac{1}{2}K>0.
$$
Indeed, we can pick $K=K(a,b,\mu)$ given by
\begin{align}\label{Kamu}
	K(a,b,\mu)=\min\bigg\{\Big(\frac{q}{8 \mu K_{opt}a^{q-\frac{N(q-2)}{2}-b}2^{\frac{N(q-2)+2b}{4}}}\Big)^{\frac{4}{N(q-2)-4+2b}},\Big(\frac{2^*_d \Lambda^{\frac{2^*_d}{2}}}{2^{\frac{2^{*}_b}{2}}8}\Big)^{\frac{2}{2^*_d-2}}\bigg\}.
	\end{align}
Moreover,  \eqref{Kamu} yields $K(a,b,\mu)\to 0$ as $\mu\to \infty$.
\end{proof}

As by-product of the last lemma is the following corollary.
\begin{corollary} There exists $K=K(a,b,\mu)>0$ such that $J>0$ on $A$. In particular, 
	\begin{align}
		J_*=\inf\Big\{J(u)\; :\; u\in S(a)\;\mbox{and}\; \|\nabla u\|_{L^2}^2=\frac{K}{2}\Big\}>0.
	\end{align}
\end{corollary}
\begin{proof} Fix $u\in A$ and let $K$ be given by \eqref{Kamu}. Arguing as in the last lemma (cf. \eqref{Fv-step1}) we can write
\begin{align*}
J(u) &\ge \frac{1}{2}\|\nabla u\|_{L^2}^{2}-  \frac{\mu}{q}K_{opt} a^{q-\frac{N(q-2)}{2}-b}(\|\nabla u\|^2_{L^2})^{\frac{N(q-2)+2b}{4}}-\frac{1}{2^*_d \Lambda^{\frac{2^*_d}{2}}}\left(\|\nabla u\|_{L^2}^2\right)^{\frac{2^*_d}{2}}\\
&\ge  \|\nabla u\|_{L^2}^{2} \Big[\frac{1}{2}-  \frac{\mu}{q}K_{opt} a^{q-\frac{N(q-2)}{2}-b}K^{\frac{N(q-2)+2b-4}{4}}-\frac{1}{2^*_d \Lambda^{\frac{2^*_d}{2}}}K^{\frac{2^*_d-2}{2}}\Big]\\
& \ge \|\nabla u\|_{L^2}^{2} \Big[\frac{1}{2}- \frac{1}{2^{\frac{2^{*}_{b}}{2}}8}- \frac{1}{2^{\frac{N(q-2)+2b}{4}}8}\Big]\\
&\ge \frac{ \|\nabla u\|_{L^2}^{2}}{4}>0.
\end{align*}
In addition, we can conclude that  $J_{*}\ge K/8$.	
\end{proof}

In what follows,  we fix $u_0 \in S(a)$ and apply Lemma \ref{vicentej} to get two numbers $s_1<0$ and $s_2>0$, of such way that the functions $u_1=\mathcal{H}(u_0,s_1)$ and $u_2=\mathcal{H}(u_0,s_2)$ satisfy
$$
\|\nabla u_1\|^2_{L^2}<\frac{K(a,b,\mu)}{2}, \; \|\nabla u_2\|_{L^2}^2>2K(a,b,\mu),\; J(u_1)>0\;\; \mbox{and} \;\; J(u_2)<0.
$$
Now, following the ideas from Jeanjean \cite{jeanjean1}, we fix the  mountain pass level given by
$$
\gamma_\mu(a,b)=\inf_{h \in \Gamma}\max_{t \in [0,1]}J(h(t))
$$
where
$$
\Gamma=\left\{h \in C([0,1],S(a)): h(0)=u_1 \;\; \mbox{and} \;\; h(1)=u_2 \right\}.
$$
From Lemma \ref{PJ1},
$$
\max_{t \in [0,1]}J(h(t))>\max \left\{J(u_1),J(u_2)\right\}>0.
$$
Therefore,
\begin{align}
	\gamma_\mu(a,b)\geq J_*>0.
\end{align}
\begin{lemma}\label{mu^*} We have that
	$\displaystyle \lim_{\mu \to \infty}\gamma_\mu(a,b)=0.$
\end{lemma}
\begin{proof}
	Indeed, fix $u_0\in S(a)$ and define the path $h_0(t)=\mathcal H(u_0,(1-t)s_1+ts_2)$. Then, from \eqref{JonH} we get
	\begin{align}\label{estgamma}
		\gamma_\mu(a,b) &\leq \max_{t\in [0,1]}J(h_0(t))\le \max_{r\geq 0}\Big\{\frac{r}{2}\|\nabla u_0\|_{L^2}^2-\frac{\mu r^{\frac{N(q-2)+2b}{4}}}{q}\|u_0\|^{q}_{L^{q}_b}\Big\}.
	\end{align}
Now, the function $r\mapsto\frac{r}{2}-\mu cr^{\frac{N(q-2)+2b}{4}}$, with  $c>0$ and $q>2+(4-2b)/N$  admits a global maximum on $(0, \infty)$ at the point
 $$r^*=\bigg(\frac{1}{\mu c}\frac{2}{N(q-2)+2b}\bigg)^{\frac{4}{N(q-2)-(4-2b)}}.$$  Consequently, from \eqref{estgamma} we obtain the existence of $C_2>0$ such that
\begin{align}
	\gamma_\mu(a,b)\leq C_2\bigg(\frac{1}{\mu}\bigg)^{\frac{4}{N(q-2)-(4-2b)}}.
\end{align}
Taking $\mu\to \infty $, it follows the result.
\end{proof}
\begin{remark} We observe that for our combined power-type nonlinearity  \eqref{f-fuction}, the general argument in proof of Proposition~2.2 in \cite{jeanjean1} remains hold. Next, we will use that Proposition without further comments.
\end{remark}
\begin{proof}[Proof of Theorem \ref{thm-HS}]
Let $(u_n)$ a $(PS)$ sequence associated with the level $\gamma_\mu(a,b)$, which is  obtained by making $u_n=\mathcal{H}(v_n,s_n)$,  where $(v_n,s_n)\in S(a)\times\mathbb{R}$ is the $(PS)$ sequence for $ \tilde{J} $ obtained by \cite[Proposition 2.2]{jeanjean1}, associated with the level $\gamma_\mu(a,b).$ More precisely, we have
\begin{equation} \label{gamma(a)}
J(u_n) \to \gamma_\mu(a,b) \quad \mbox{as} \quad n \to \infty,
\end{equation}
and
\begin{equation*} \label{der1}
\|J^{\prime}_{|_{S(a)}}(u_n)\| \to 0 \quad \mbox{as} \quad n \to \infty.
\end{equation*}
Setting the functional $\psi:H^{1}(\mathbb{R}^N) \to \mathbb{R}$ given by
$$
\psi(u)=\frac{1}{2}\int_{\mathbb{R}^N}|u|^2\,dx,
$$
it follows that $S(a)=\psi^{-1}(\{a^2/2\})$. Then, by Willem \cite[Proposition 5.12]{Willem}, there exists $(\lambda_n) \subset \mathbb{R}$ such that
$$
\|J'(u_n)-\lambda_n{\psi}^{\prime}(u_n)\|_{H^{-1}} \to 0 \quad \mbox{as} \quad n \to \infty.
$$
Hence, for any $\varphi\in H^{1}(\mathbb{R}^N)$
\begin{equation}\label{eqaprox}
    \int_{\mathbb{R}^N}\nabla u_n\nabla \varphi dx-\int_{\mathbb{R}^N}f(x,u_n)\varphi dx-\lambda_n\int_{\mathbb{R}^N}u_n\varphi dx=o_n(1)\|\varphi\|.
\end{equation}
Arguing as in \cite[Proposition~1]{Bere-Lions} and  \cite[Lemma~1.1]{FigueiredoNussbaum}, we can see that if $u\in H^{1}(\mathbb{R}^N)$ is a solution of
$$-\Delta u=g(x,u)\ \mbox{in}\ \R^N,$$
then the following identity holds
\begin{equation}\label{Pohozaev1}
 \frac{N-2}{2}\int_{\R^N} |\nabla u|^2 dx= N \int_{\R^N} G(x,u) dx+ \sum_{i=1}^N \int_{\R^N}  x_i G_{x_i} (x,u) dx,
 \end{equation}
 where $G(x,t)=\int_{0}^{t}g(x,\tau)d\tau$. Then
 \begin{equation}\label{Pohozaev3}
 \frac{N}{2}\int_{\R^N} g(x,u) u dx- \int_{\R^N} |\nabla u|^2 dx= N \int_{\R^N} G(x,u) dx+ \sum_{i=1}^N \int_{\R^N}  x_i G_{x_i} (x,u) dx.
 \end{equation}
 If $G(x,t)=\dfrac{\mu}{q}|x|^{-b} |t|^{ q}+\frac{1}{2^*_{d}}|x|^{-d}|t|^{2^{*}_d} -\dfrac{\lambda }{2}t^2$ in \eqref{Pohozaev3}, it follows that
  \begin{equation}\label{Pohozaev2}
\int_{\R^N} |\nabla u|^2 dx+ \mu\Big(\frac{N-b}{q}-\frac{N}{2} \Big)\int_{\R^N} |x|^{-b} |u|^{q}  dx-\int_{\mathbb{R}^N}|x|^{-d}|u|^{2^{*}_d}dx=0.
 \end{equation}
Motivated by the equality above we define $Q:H^{1}(\mathbb{R}^N)\to\mathbb{R}$ by 
\begin{equation}\nonumber
  Q(u)=  \|\nabla u\|^{2}_{L^2}+ \mu\Big(\frac{N-b}{q}-\frac{N}{2} \Big)\|u\|^{q}_{L^{q}_b}-\|u\|^{2^{*}_d}_{L^{2^{*}_d}_d}.
\end{equation}
Another important limit involving the sequence $(u_n)$ is
\begin{equation} \label{EQ1**}
Q(u_n)=\|\nabla u_n\|^{2}_{L^2}+ \mu\Big(\frac{N-b}{q}-\frac{N}{2} \Big)\|u_n\|^{q}_{L^{q}_b}-\|u_n\|^{2^{*}_d}_{L^{2^{*}_d}_d}\to 0,
\end{equation}
as $n\to\infty$,  where $u_n=\mathcal{H}(v_n,s_n),$ which is obtained using the limit below
$$
{\partial_s}\tilde{J}(v_n,s_n) \to 0 \quad \mbox{as} \quad n \to \infty,
$$
which can be proved by using \cite[Proposition 2.2]{jeanjean1}.

Arguing as in \cite[Lemmas 2.3, 2.4 and 2.5]{jeanjean1}, we know that $(u_n)$ is a bounded sequence in $H^{1}(\mathbb{R}^N)$, and so, the number $\lambda_n$  must satisfy the equality below
$$
\lambda_n=\frac{1}{\|u_n\|^{2}_{L^2}}\bigg\{ \|\nabla u_n\|^{2}_{L^2} -\int_{\mathbb{R}^N} f(x,u_n) u_n dx \bigg\}+o_n(1),
$$
From \eqref{L2-H} we have $\|u_n\|^{2}_{L^2}=\|\mathcal{H}(v_n,s_n)\|^{2}_{L^2}=\|v_n\|^{2}_{L^2}=a^2$. Thus, we also can write
\begin{equation} \label{lambdan}
\begin{aligned}
	a^2\lambda_n & = \|\nabla u_n\|^{2}_{L^2} -\int_{\mathbb{R}^N} f(x,u_n) u_n dx +o_n(1)\\
 &= \|\nabla u_n\|^{2}_{L^2}-\mu\|u_n\|^{q}_{L^{q}_b}-\|u_n\|^{2^{*}_d}_{L^{2^{*}_d}_d}+o_n(1).
 \end{aligned}
\end{equation}
Since $(u_n)$ is a bounded in $H^{1}(\mathbb{R}^N)$, the continuous  embeddings \eqref{sub-emb} and \eqref{GNopt*} imply $(\lambda_n)$ is bounded, so that, $\lambda_n \rightarrow \lambda.$  On the other hand, we have $u_n \rightharpoonup u$ weakly in $H^{1}(\mathbb{R}^N)$  and from the compact embedding $H^1(\mathbb{R}^N)\hookrightarrow L^{q}_b(\Real^N)$ we get
$\|u_n\|^{q}_{L^q_b}\rightarrow \|u\|^{q}_{L^q_b}.$ By using the definition of $Q(u_n)$, from \eqref{lambdan} we have 
\begin{equation} \nonumber
\begin{aligned}
    a^2\lambda_n 
    &= Q(u_n)-\mu\Big(\frac{N-b}{q}-\frac{N}{2}+1 \Big)\|u_n\|^{q}_{L^{q}_b}+o_n(1)\\
    &= Q(u_n)+\mu\frac{N-2}{2q}\Big(q-2^*_b\Big)\|u_n\|^{q}_{L^{q}_b}+o_n(1).
\end{aligned}
\end{equation}
From \eqref{EQ1**}, letting $n\to\infty$
\begin{equation} \nonumber
\begin{aligned}
    a^2\lambda=\mu\frac{N-2}{2q}\Big(q-2^*_b\Big)\|u\|^{q}_{L^{q}_b}.
\end{aligned}
\end{equation}
Since $q<2^*_b$, it is immediate that $\lambda\leq 0$. We claim that $u\not\equiv 0$ and  $u_n\to u$ strongly in $H^1(\Real^N)$ if $\mu^*>0$ is chosen such that \begin{align}\label{mu*}
	\gamma_\mu(a,b)<\frac{2-d}{N-d}\frac{\Lambda^{\frac{N-d}{2-d}}}{2},\,\,\,\,\mbox{for all }\mu\geq \mu^*.
	\end{align}
Thanks to Lemma \ref{mu^*}, such a choice is possible.  We argue by contradiction. Suppose that $u\equiv 0$. Then, $u_n\to 0$ in $L^{q}_b(\Real^N)$, $Q(u_n)\to 0$ and there exists $m\in \Real$ such that $\|\nabla u_n\|_{L^2}^2\to m$ as $n\to \infty$. Thereby, from this and by \eqref{critenergcnst} and \eqref{EQ1**} we deduce that 
\begin{align*}
	m+o_n(1)&=\|\nabla u_n\|_{L^2}^2\geq\Lambda\left(\int_{\R^N}|x|^{-d}|u_n|^{2_d^*}\,dx\right)^{\frac{2}{2^*_d}}\\
	&=\Lambda\left(\|\nabla u_n\|^{2}_{L^2}+ \mu\Big(\frac{N-b}{q}-\frac{N}{2} \Big)\|u_n\|^{q}_{L^{q}_b}-Q(u_n)\right)^{\frac{2}{2^*_d}}\\
	&=\Lambda m^{\frac{2}{2^*_d}}+o_n(1).
\end{align*}   
We have either $m=0$ or $m\geq \Lambda^{\frac{N-d}{2-d}}$. If $m=0$, from \eqref{GNopt*} we obtain $J(u_n)\to 0$ as $n\to \infty$, which is a contradiction. If instead $m\geq \Lambda^{\frac{N-d}{2-d}}$, it follows that
\begin{align*}
	\gamma_\mu(a,b)& =\lim_{n\to\infty}J(u_n)\\
 &=\lim_{n\to\infty}\Big[\frac{Q(u_n)}{2^*_d}+\frac{2-d}{N-d}\frac{\|\nabla u_n\|_{L^2}^2}{2}+o_n(1)\Big]\\
 &=\frac{2-d}{N-d}\frac{m}{2}\geq \frac{2-d}{N-d}\frac{\Lambda^{\frac{N-d}{2-d}}}{2},
\end{align*}
which gives again a contradiction from \eqref{mu*}. Thus, $u\not\equiv 0$. Furthermore, by the weak convergence we deduce from \eqref{eqaprox} that $u$ is a solution non-null to equation
\begin{align}\label{eq}
	-\Delta u-f(x,u)=\lambda u\;\;\mbox{in}\;\;\Real^N
\end{align}
and from Pohozaev identity \eqref{Pohozaev2} we have $Q(u)=0$ and from \eqref{eq} 
\begin{equation}\label{J(u)>0}
    J(u)=\frac{\lambda}{2}\|u\|^{2}_{L^2}+\mu\Big(\frac{1}{2}-\frac{1}{q}\Big)\|u\|^{q}_{L^{q}_b}+\Big(\frac{1}{2}-\frac{1}{2^{*}_d}\Big)\|u\|^{2^*_d}_{L^{2^{*}_d}_d}>0
\end{equation}
for $\mu>0$ large enough. Now we will show that $u_n\to u$ in $H^1(\mathbb{R}^N)$. Defining $v_n:=u_n-u$, then up to a subsequence, there exists $l\in \Real$ such that $\displaystyle\lim_{n\to \infty }\|\nabla v_n\|_{L^2}^2=l.$ Note that $v_n\rightharpoonup 0$ in $H^1(\mathbb{R}^N)$, $u_n\to u$ in $L^{q}_b(\Real^N)$,
\begin{align}\label{grad}
	\|\nabla u_n\|_{L^2}^2&=\|\nabla u\|_{L^2}^2+\|\nabla v_n\|_{L^2}^2+o_n(1),
\end{align}
and by Brezis-Lieb Lemma \cite{Brezis-Lieb} 
\begin{align}\label{BLL}
	\|u_n\|^{2^{*}_d}_{L^{2^{*}_d}_d}=	\|u\|^{2^{*}_d}_{L^{2^{*}_d}_d}+\|v_n\|^{2^{*}_d}_{L^{2^{*}_d}_d}+o_n(1).
\end{align}
Thus, since  $Q(u_n)\to 0$ and $u_n\to u$ in $L^{q}_b(\Real^N)$ we get
\begin{align*}
	\|\nabla u\|_{L^2}^2+\|\nabla v_n\|_{L^2}^2 &=Q(u_n)-\mu\Big(\frac{N-b}{q}-\frac{N}{2} \Big)\|u_n\|^{q}_{L^{q}_b}+\|u\|^{2^{*}_d}_{L^{2^{*}_d}_d}+\|v_n\|^{2^{*}_d}_{L^{2^{*}_d}_d}+o_n(1)\\
 &=-\mu\Big(\frac{N-b}{q}-\frac{N}{2} \Big)\|u\|^{q}_{L^{q}_b}+\|u\|^{2^{*}_d}_{L^{2^{*}_d}_d}+\|v_n\|^{2^{*}_d}_{L^{2^{*}_d}_d}+o_n(1).
\end{align*}
That is,
\begin{align*}
	Q(u)+\|\nabla v_n\|_{L^2}^2& = \|v_n\|^{2^{*}_d}_{L^{2^{*}_d}_d}+o_n(1)
\end{align*}
Consequently, since that $Q(u)=0$ we obtain 
\begin{align}\label{l}
	\lim_{n\to \infty }\|\nabla v_n\|_{L^2}^2=\lim_{n\to \infty }\|v_n\|^{2^{*}_d}_{L^{2^{*}_d}_d}=l,
\end{align}
and thus,
either $l=0$ or
\begin{align}\label{Lambdaineq}
	\Lambda\leq \lim_{n\to \infty}\frac{\|\nabla v_n\|_{L^2}^2}{\|v_n\|^{2}_{L^{2^{*}_d}_d}}=l^\frac{2-d}{N-d}.
\end{align}
If it occurs \eqref{Lambdaineq}, then from \eqref{grad}, \eqref{BLL} and \eqref{l} 
\begin{align*}
	\gamma_\mu(a,b)+o_n(1)&=\frac{1}{2}\|\nabla u_n\|_{L^2}^2-\frac{\mu}{q}\int_{\R^N} |x|^{-b}|u_n|^{q}\,dx-\frac{1}{2_d^*}\int_{\R^N} |x|^{-d}|u_n|^{2^*_d}\,dx\\
	&=J(u)+\frac{1}{2}\|\nabla v_n\|_{L^2}^2-\frac{1}{2_d^*}\int_{\R^N} |x|^{-d}|v_n|^{2^*_d}\,dx+o_n(1)\\
	&\geq J(u)+\frac{2-d}{N-d}\frac{\Lambda^{\frac{N-d}{2-d}}}{2}+o_n(1).
\end{align*}
Hence, follows by \eqref{mu*} that
\begin{align}\nonumber
		\gamma_\mu(a,b)\geq J(u)+\gamma_\mu(a,b)
\end{align}
which is a contradiction since $J(u)>0$. Therefore, \eqref{l} yields
$\|\nabla v_n\|_{L^2}\to 0$ 
and also $v_n\to 0$ in $L^{q}_b(\Real^N)$ and $L^{2^*_d}_d(\Real^N)$. By using the test function $\varphi=u_n-u$ in  \eqref{eqaprox} and \eqref{eq} and subtracting the obtained expressions we can write
\begin{align*}
    &\int_{\mathbb{R}^N}|\nabla(u_n-u)|^2dx-\mu \int_{\mathbb{R}^N}|x|^{-b}\big(|u_n|^{q-2}u_n-|u|^{q-2}u\big)(u_n-u)dx\\
    &=\int_{\mathbb{R}^N}|x|^{-d}\big(|u_n|^{2^{*}_d-2}u_n-|u|^{2^{*}_d-2}u\big)(u_n-u)dx+\int_{\mathbb{R}^N}(\lambda_n u_n-\lambda u)(u_n-u)dx+o_n(1).
\end{align*}
Letting $n\to \infty$ we obtain 
\begin{align*}
   0=\lim_{n\to\infty}\int_{\mathbb{R}^N}(\lambda_n u_n-\lambda u)(u_n-u)dx=\lambda\lim_{n\to\infty}\int_{\mathbb{R}^N}(u_n-u)^2dx.
\end{align*}
Hence, we conclude that $v_n\to 0 $ in $L^2(\Real^N)$ and thus $u_n\to u$ in $H^1(\mathbb{R}^N)$ which ensures $u\in S(a)$. 
\end{proof}
\section{Normalized solutions: the singular exponential critical growth}
\label{sec4}
In this section we are assuming $N=2$, $0<b<2$, $a\in (0,1)$ and that $f$ has  (singular) critical Trudinger-Moser growth \eqref{TM-growth} and satisfying $(f_1)$-$(f_3)$. By $(f_1)$-$(f_3)$ and 
\eqref{TM-growth}, fix $q>2$, for any $\zeta>0$ and $\alpha>4\pi$, there exists $D=D(q, \alpha,\mu, b)>0$
\begin{equation}\label{f<exp}
|f(t)|\le \zeta|t|^{\tau}+D|t|^{q-1}\big(e^{\alpha(1-\frac{b}{2})t^2}-1\big),\;\;\forall\; t\in\mathbb{R}
\end{equation}
and
\begin{equation}\label{F<exp}
|F(t)|\le \zeta|t|^{\tau+1}+D|t|^{q}\big(e^{\alpha(1-\frac{b}{2})t^2}-1\big),\;\;\forall\; t\in\mathbb{R}.
\end{equation}
In particular, from \eqref{f<exp}
\begin{equation}\label{ft<exp}
|f(t)t|\le \zeta|t|^{\tau+1}+D|t|^{q}\big(e^{\alpha(1-\frac{b}{2})t^2}-1\big),\;\;\forall\; t\in\mathbb{R}.
\end{equation}
Now, let us recall the following particular case of singular  Trudinger–Moser inequality, see for instance \cite[Theorem~C]{LamLuJDE}  and \cite[Theorem~1.1]{LamLuANS,LamLuIbero}.
\begin{lemma} \label{Ibero} Let $0\le b<2$. Then,
\begin{enumerate}
\item [$(i)$] For any $\alpha>0$ and $u\in H^{1}(\mathbb{R}^2)$ there holds $$(e^{\alpha u^2}-1)\in L^{1}_{b}(\mathbb{R}^2).$$
\item [$(ii)$]  For $0\le \alpha< 4\pi$, denote 
\begin{equation}\nonumber
\mathrm{AT}(\alpha, b)=\sup_{\|\nabla u\|_{L^2}\le 1}\frac{1}{\|u\|^{2-b}_{L^2}}\|(e^{\alpha(1-\frac{b}{2})u^2}-1)\|_{L^{1}_{b}}.
\end{equation}
Then, there exist positive constants $c=c(b)$ and $C=(b)$ such that, when $\alpha$ is close enough to $4\pi$
\begin{equation}\nonumber
\frac{c(b)}{\big(1-\frac{\alpha}{4\pi}\big)^{\frac{2-b}{2}}}\le \mathrm{AT}(\alpha, b)\le \frac{C(b)}{\big(1-\frac{\alpha}{4\pi}\big)^{\frac{2-b}{2}}}.
\end{equation}
Moreover, the constant $4\pi$ is sharp in the sense that $\mathrm{AT}(4\pi, b)=\infty.$
\end{enumerate}
\end{lemma}
\begin{lemma}\label{lemma-conv}
Let $0<b<2$ and let $\{u_n\}$ be a sequence in $H^{1}(\mathbb{R}^2)$ with $u_n\in S(a)$ and
\begin{equation}
\limsup_{n
\to\infty}\|\nabla u_n\|^{2}_{L^2 }<1-a^2.
\end{equation}
Then
\begin{enumerate}
\item [$(a)$] there exist $t>1$ and $C=C(t,a,b)>0$ such that
\begin{equation}\nonumber
\|\big(e^{4\pi(1-\frac{b}{2})u^{2}_n}-1\big)\|^{t}_{L^t_{b}}\le C,\;\;\forall\; n\in\mathbb{N};
\end{equation}
\item [$(b)$] if $u_n\rightharpoonup u$ in $H^{1}(\mathbb{R}^2)$ and $u_n(x)\to u(x)$ a.e in $\mathbb{R}^2$, then 
\begin{enumerate}
\item [$(i)$]  there exists $\alpha>4\pi$ such that for all $q>2$
\begin{equation}\nonumber
|u_n|^{q}\big(e^{\alpha(1-\frac{b}{2})u^2_n}-1\big) \to |u|^{q}\big(e^{\alpha(1-\frac{b}{2})u^2}-1\big)  \;\;\mbox{in}\;\;L^{1}_{b}(\mathbb{R}^2).
\end{equation}
\item [$(ii)$] $$F(u_n)\to F(u) \;\;\mbox{and} \;\; f(u_n)u_n\to f(u)u\;\;\mbox{in} \;\; L^{1}_{b}(\mathbb{R}^2).$$
\end{enumerate}
\end{enumerate}
\end{lemma}
\begin{proof}
$(a)$ We have
\begin{equation}\nonumber
\limsup_{n
\to\infty}\|\nabla u_n\|^{2}_{L^2 }<1-a^2\;\;\mbox{and}\;\; \|u_n\|^2_{L^2}=a^2<1.
\end{equation}
Hence, there is $\eta>0$ and $n_0\in\mathbb{N}$ such that 
\begin{equation}\label{norm-sub}
\|u_n\|^2<\eta<1,\;\; \mbox{for any}\;\; n\ge n_0.
\end{equation}
We recall that $(e^{s}-1)^{t}\le e^{ts}-1$ for $t>1$ and $s\ge 0$. So,  for $t>1 $ such that $t\eta<1$
\begin{equation}\nonumber
\int_{\mathbb{R}^2}\big(e^{4\pi(1-\frac{b}{2})u^{2}_n}-1\big)^{t}\frac{dx}{|x|^b}\le \int_{\mathbb{R}^2}\big(e^{4t\eta\pi(1-\frac{b}{2})\big(\frac{|u_n|}{\|u_n\|}\big)^{2}}-1\big)\frac{dx}{|x|^b},\;\; \mbox{for any}\;\; n\ge n_0.
\end{equation}
Thus, from Lemma~\ref{Ibero}, there is $C=C(b)>0$ with
\begin{equation}\nonumber
\int_{\mathbb{R}^2}\big(e^{4\pi(1-\frac{b}{2})u^{2}_n}-1\big)^{t}\frac{dx}{|x|^b}\le C_1=C_1(t,\eta,a,b)=\frac{C(b)a^{2-b}}{(1-t\eta)^{\frac{2-b}{2}}},\;\; \mbox{for any}\;\; n\ge n_0.
\end{equation}
Now, we can pick
\begin{equation}
C=\max\Big\{C_1, \|\big(e^{4\pi(1-\frac{b}{2})u^{2}_1}-1\big)\|_{L^{t}_b}^{t},\cdots,  \|\big(e^{4\pi(1-\frac{b}{2})u^{2}_{n_0}}-1\big)\|_{L^{t}_b}^{t}\Big\}.
\end{equation}
$(b)$  First, we recall that according to \cite[Theorem~1.2]{CaLu}  we have the compact embedding 
\begin{equation}\label{n=2compact}
H^{1}(\mathbb{R}^2)\hookrightarrow L^{q}_{b}(\mathbb{R}^2), \;\;\mbox{for any}\;\; q\ge 2\;\;\mbox{and}\;\; 0<b<2.
\end{equation}
Hence, up to a subsequence, we can assume that
\begin{equation}\label{lp-conv}
u_n\to u \;\;\mbox{in}\;\;L^{q}_{b}(\mathbb{R}^2),q\ge 2.
\end{equation}
$(i)$ Let $\eta>0$ be such as in \eqref{norm-sub}. Since $\eta<1$, for $\alpha>4\pi$ with $\alpha$ close to $4\pi$ we can take $t>1$ such that $t \alpha \eta <4\pi$. Set 
\begin{align*}
h_n(x)=e^{\alpha(1-\frac{b}{2})u^{2}_n(x)}-1.
\end{align*}
We have 
\begin{align*}
h_n(x)\to  h(x)=e^{\alpha(1-\frac{b}{2})u^{2}(x)}-1  \;\;\mbox{a.e in }\;\; \mathbb{R}^2.
\end{align*}
Also, arguing as in the proof of item $(a)$, we get 
\begin{align*}
h_n\in L^{t}_{b}(\mathbb{R}^2) \;\;\mbox{and}\;\; \sup_{n}\|h_n\|_{L^{t}_b}<\infty.
\end{align*}
So, up to a subsequence 
\begin{align*}
  h_n\rightharpoonup h \;\;\mbox{in }
 L^{t}_{b}(\mathbb{R}^2).
\end{align*}
From \eqref{lp-conv} we get 
\begin{align*}
|u_n|^q\to |u|^q\;\;\mbox{in }
 L^{\frac{t}{t-1}}_{b}(\mathbb{R}^2).
\end{align*}
Thus,  the Riesz representation theorem yields
\begin{align*}
\lim_{n\to \infty}\int_{\mathbb{R}^2}|u_n|^qh_n\frac{dx}{|x|^b}=\int_{\mathbb{R}^2}|u|^qh\frac{dx}{|x|^b}.
\end{align*}
Since 
\begin{align*}
|u_n|^qh_n\ge 0\;\;\mbox{and}\;\; |u|^qh\ge 0
\end{align*}
we obtain 
\begin{align*}
\||u_n|^qh_n\|_{L^1_{b}}\to \||u|^qh\|_{L^1_{b}}.
\end{align*}
Since $|u_n(x)|^qh_n(x)\to |u(x)|^qh(x)$ a.e in $\mathbb{R}^2$ we conclude 
\begin{equation}\nonumber
|u_n|^{q}h_n \to |u|^{q}h \;\;\mbox{in}\;\;L^{1}_{b}(\mathbb{R}^2).
\end{equation}
$(ii)$ For  $\alpha>4\pi$ close to $4\pi$  and $q>2$ as in the item~$(i)$, by \eqref{F<exp} we can write 
\begin{align*}
|F(u_n)|\le \zeta|u_n|^{\tau+1}+D|u_n|^{q}\big(e^{\alpha(1-\frac{b}{2})u_n^2}-1\big).
\end{align*}
By  item~$(i)$ and  \eqref{lp-conv}  we have
\begin{equation}\nonumber
|u_n|^{q}\big(e^{\alpha(1-\frac{b}{2})u^2_n}-1\big) \to |u|^{q}\big(e^{\alpha(1-\frac{b}{2})u^2}-1\big)  \;\;\mbox{in}\;\;L^{1}_{b}(\mathbb{R}^2)
\end{equation}
and 
\begin{equation}\nonumber
u_n\to u \;\;\mbox{in}\;\;L^{\tau+1}_{b}(\mathbb{R}^2).
\end{equation}
Now, we can apply the generalized Lebesgue dominated convergence theorem to derive that
$$F(u_n)\to F(u)\;\;\mbox{in} \;\; L^{1}_{b}(\mathbb{R}^2).$$
From \eqref{ft<exp} and  item~$(i)$, by similar argument  one has
$$f(u_n)u_n\to f(u)u\;\;\mbox{in} \;\; L^{1}_{b}(\mathbb{R}^2).$$
\end{proof}
\subsection{The minimax strategy}
In order to get our existence results  we will apply the minimax approach. We will use the same notations:
\begin{enumerate}
\item $S(a)=\left\{u\in H^{1}(\mathbb{R}^2)\; :\; \|u\|_{L^2}=a \right\}$, where $ \|\quad\|_{L^p}$ denotes the usual norm of the Lebesgue space $L^{p}(\mathbb{R}^{2})$, $p\in [1, +\infty)$\\
\item  $J:  H^{1}(\mathbb{R}^2)\to \mathbb{R}$ given by
\begin{equation}\nonumber
J(u)=\frac{1}{2}\int_{\mathbb{R}^{2}}|\nabla u|^{2}d x-\int_{\mathbb{R}^2}|x|^{-b}F(u)dx
\end{equation}
\item $\mathcal{H}: H= H^{1}(\mathbb{R}^2)\times \mathbb{R}\to  H^{1}(\mathbb{R}^{2})$ with
\begin{equation}\nonumber
\mathcal{H}(u,s)=e^{s}u(e^{s}x)
\end{equation}
\item  $\tilde{J}: H=H^{1}(\mathbb{R}^2)\times \mathbb{R} \to \mathbb{R}$ is defined by
\begin{equation}\nonumber
\begin{aligned}
\tilde{J}(u,s)=\frac{e^{2s}}{2}\int_{\mathbb{R}^{2}}|\nabla u|^{2} d x-\frac{1}{e^{(2-b)s}}\int_{\mathbb{R}^2}|x|^{-b}F(e^{s}u)d x=J(\mathcal{H}(u,s)).
\end{aligned}
\end{equation}
\end{enumerate}
Note that
\begin{equation}\label{N=2grad-H}
   \int_{\mathbb{R}^2}|\nabla\mathcal{H}(u,s)|^{2}dx=e^{2s}\int_{\mathbb{R}^N}|\nabla u|^{2}dx.
\end{equation}
and 
\begin{equation}\label{N=2Lp-H}
    \int_{\mathbb{R}^2}|x|^{-b}|\mathcal{H}(u,s)|^{\xi}dx=e^{(\xi-2+b)s}\int_{\mathbb{R}^2}|x|^{-b}|u|^{\xi}d x,\;\;\mbox{for all}\;\; \xi\ge 2 \;\;\mbox{and}\;\; 0\le b<2.
\end{equation}
\begin{lemma}\label{vicentejTM}
Let $u\in S(a)$  be arbitrary but fixed. Then we have:  
\begin{itemize}
\item [$(i)$]  If $s\to -\infty$, then  $\|\nabla\mathcal{H}(u,s)\|_{L^2}\to 0$ and $J(\mathcal{H}(u,s))\to 0$.
\item [$(ii)$] If $s\to +\infty$, then  $\|\nabla\mathcal{H}(u,s)\|_{L^2}\to +\infty$ and $J(\mathcal{H}(u,s))\to -\infty$.
\end{itemize}
\end{lemma}
\begin{proof}
From \eqref{N=2grad-H} and \eqref{N=2Lp-H} with $\xi>2$ we obtain 
\begin{equation}\label{TMgradH}
\|\nabla\mathcal{H}(u,s)\|_{L^2}\to 0 \;\;\mbox{and}\;\; \|\mathcal{H}(u,s)\|_{L^{\xi}_b}\to 0\;\;\mbox{as}\;\; s\to -\infty.
\end{equation}
Thus, since $a\in (0,1)$  and, from \eqref{N=2Lp-H} with $b=0$ and $\xi =2$,  $\|\mathcal{H}(u,s)\|_{L^{2}}=\|u\|_{L^{2}}=a$ there are there are $s_0< 0$ and $\eta \in (0,1)$ such that
\begin{align*}
\|\mathcal{H}(u,s)\|^2\le \eta<1\;\;\mbox{for all}\;\; s\in (-\infty, s_0].
\end{align*}
Arguing  in Lemma~\ref{lemma-conv}-$(a)$,  for $\alpha>4\pi$ with $\alpha$ close to $4\pi$ we can take $t>1$ close to $1$ such that $t \alpha \eta <4\pi$ and 
\begin{equation}\nonumber
\int_{\mathbb{R}^2}\big(e^{\alpha(1-\frac{b}{2})\mathcal{H}^2(u,s)}-1\big)^{t}\frac{dx}{|x|^b}\le \int_{\mathbb{R}^2}\big(e^{t\alpha\eta(1-\frac{b}{2})\big(\frac{|\mathcal{H}(u,s)|}{\|\mathcal{H}(u,s)\|}\big)^{2}}-1\big)\frac{dx}{|x|^b},\;\; \mbox{for any}\;\; s\in (-\infty, s_0].
\end{equation}
Thus, from Lemma~\ref{Ibero}, there is $C=C(t,\eta,a,b)>0$ with
\begin{equation}\nonumber
\|\big(e^{\alpha(1-\frac{b}{2})\mathcal{H}^2(u,s)}-1\big)\|_{L^{t}_{b}}\le  C,\;\; \mbox{for any}\;\; s\in (-\infty, s_0].
\end{equation}
 Now, from \eqref{F<exp} and  the Hölder’s inequality, for any $s\in (-\infty, s_0]$
 \begin{align*}
 \|F(\mathcal{H}(u,s))\|_{L^{1}_{b}}& \le \zeta \|\mathcal{H}(u,s)\|^{\tau+1}_{L^{\tau+1}_{b}}+D\||\mathcal{H}(u,s)|^{q}\big(e^{\alpha(1-\frac{b}{2})\mathcal{H}^2(u,s)}-1\big)\|_{L^1_{b}}\\
 & \le \zeta \|\mathcal{H}(u,s)\|^{\tau+1}_{L^{\tau+1}_{b}}+CD\|\mathcal{H}(u,s)\|^q_{L^{\frac{qt}{t-1}}_b}.
 \end{align*}
 By \eqref{TMgradH}
 \begin{align*}
\|\nabla\mathcal{H}(u,s)\|_{L^2}\to 0 \;\;\mbox{and}\;\;  \|F(\mathcal{H}(u,s))\|_{L^{1}_{b}}\to 0, \;\;\mbox{as}\;\; s\to -\infty.
 \end{align*}
 from where it follows that
  \begin{align*}
 J(\mathcal{H}(u,s))\to 0, \;\;\mbox{as}\;\; s\to -\infty.
 \end{align*}
This proves $(i)$. To  show $(ii)$, note that \eqref{N=2grad-H}  yields
 \begin{align*}
\|\nabla\mathcal{H}(u,s)\|_{L^2}\to +\infty  \;\;\mbox{as}\;\; s\to +\infty.
 \end{align*}
 In addition, the assumption $(f_3)$,  \eqref{N=2grad-H}  and \eqref{N=2Lp-H} ensures
 \begin{equation}\label{F-low}
 \begin{aligned}
 J(\mathcal{H}(u,s))& \le \frac{1}{2}\|\nabla \mathcal{H}(u,s)\|^2_{L^2}-\frac{\mu}{p}\|\mathcal{H}(u,s)\|^{p}_{L^{p}_b}\\
 &= \frac{e^{2s}}{2}\|\nabla u\|^2_{L^2}-\frac{\mu}{p}e^{(p-2+b)s}\|u\|^{p}_{L^{p}_b}\\
 &\to -\infty\;\;\mbox{as}\;\; s\to +\infty,
 \end{aligned}
 \end{equation}
 where we have used  $p> 4$ in $(f_3)$.
\end{proof}
\begin{lemma}\label{L=A<B}
There exists $K=K(a,b,\mu)>0$ small enough satisfying
\begin{equation}\label{JA<JB}
0<\sup_{u\in A}J(u)<\inf_{u\in B}J(u),
\end{equation}
where 
\begin{align*}
A=\big\{u\in S(a)\; : \;  \|\nabla u\|^2_{L^2}\le K\big\}\;\;\mbox{and}\;\; B=\big\{u\in S(a)\; : \;  \|\nabla u\|^2_{L^2}=2K\big\}.
\end{align*}
Moreover, $K(a,b,\mu)\to 0$ when $\mu\to \infty$.
\end{lemma}
\begin{proof}
From \cite{CKN},  we derive the following Gagliardo–Sobolev type inequality
\begin{equation}\label{GSN=2}
\|u\|_{L^{\xi}_b}\le C(\xi, b)\|\nabla u\|^{\gamma}_{L^2}\|u\|^{1-\gamma}_{L^2}
\end{equation}
where $\gamma=1-\frac{2}{\xi}(1-\frac{b}{2})$, with $\xi>2$.  For $u\in S(a)$ with $\|\nabla u\|^2_{L^2}<1-a^2$, from \eqref{F<exp}, Lemma~\ref{Ibero} and  the Hölder’s inequality
\begin{equation}\label{F<Lps}
 \|F(u)\|_{L^{1}_{b}} \le \zeta \|u\|^{\tau+1}_{L^{\tau+1}_{b}}+D_1\|u\|^q_{L^{qt^{\prime}}_b},
\end{equation}
where $q>2$, $t^{\prime}=\frac{t}{t-1}$, $t>1$ closed to $1$, $\zeta>0$ does not depend on $\mu$, while $D_1$ depends on $\mu$.  By using \eqref{GSN=2} and \eqref{F<Lps}
\begin{equation}\label{F<grads}
\begin{aligned}
 \|F(u)\|_{L^{1}_{b}} & \le \zeta \|u\|^{\tau+1}_{L^{\tau+1}_{b}}+D_1\|u\|^q_{L^{\frac{qt}{t-1}}_b}\\
 &  \le \zeta C(\tau+1,b)^{\tau+1}\|\nabla u\|^{(\tau+1)(1-\frac{2}{\tau+1}(1-\frac{b}{2}))}_{L^2}\|u\|^{(\tau+1)(\frac{2}{\tau+1}(1-\frac{b}{2}))}_{L^2}\\
 &+D_1C^{q}(qt^{\prime},b)\|\nabla u\|^{q(1-\frac{2}{qt^{\prime}}(1-\frac{b}{2}))}_{L^2}\|u\|^{q\frac{2}{qt^{\prime}}(1-\frac{b}{2})}_{L^2}\\
 & =C_1\big(\|\nabla u\|^{2}_{L^2}\big)^{\frac{\tau-1+b}{2}}a^{2-b}+D_2\big(\|\nabla u\|^{2}_{L^2}\big)^{\frac{q}{2}-\frac{1}{t^{\prime}}(1-\frac{b}{2})}a^{\frac{2}{t^{\prime}}(1-\frac{b}{2})},
\end{aligned}
\end{equation}
where $C_1=\zeta C(\tau+1,b)^{\tau+1}$ and  $D_2=D_1C^{q}(qt^{\prime},b)$. Since $F\ge0$ on $\mathbb{R}$, for any $u,v\in H^{1}(\mathbb{R}^2)$ we have
\begin{equation}\label{JvJu}
J(v)-J(u)\ge \frac{1}{2}\|\nabla v\|^2_{L^2}-\frac{1}{2}\|\nabla u\|^2_{L^2}- \|F(v)\|_{L^{1}_{b}}.
\end{equation}
Let $0<K=K(a,b,\mu)<\frac{1-a^2}{2}$  to be determined later.  For $\|\nabla u\|^2_{L^2}\le K$ and $\|\nabla v\|^2_{L^2}=2K$,  \eqref{F<grads} and \eqref{JvJu} yield
\begin{equation*}
\begin{aligned}
J(v)-J(u)& \ge \frac{1}{2} K-C_1(2K)^{\frac{\tau-1+b}{2}}a^{2-b}-D_2(2K)^{\frac{q}{2}-\frac{1}{t^{\prime}}(1-\frac{b}{2})}a^{\frac{2}{t^{\prime}}(1-\frac{b}{2})}\\
&=K\left[\frac{1}{2}-C_12^{\frac{\tau-1+b}{2}}K^{\frac{\tau-1+b}{2}-1}a^{2-b}-D_22^{\frac{q}{2}-\frac{1}{t^{\prime}}(1-\frac{b}{2})}K^{\frac{q}{2}-\frac{1}{t^{\prime}}(1-\frac{b}{2})-1}a^{\frac{2}{t^{\prime}}(1-\frac{b}{2})}\right]. 
\end{aligned}
\end{equation*}
Since $0<b<2$, $q>2$, $\tau>3$ and $t^{\prime}=\frac{t}{t-1}$ with $t$ closed to $1$ we have $\tau-1+b>2$ and $\frac{q}{2}-\frac{1}{t^{\prime}}(1-\frac{b}{2})>1$. So, we may choose
\begin{equation}\label{K-defTM}
K(a,b,\mu)=\min\left\{\frac{1-a^2}{4},\left[ \frac{1}{2^{\frac{\tau-1+b}{2}}8C_1a^{2-b}}\right]^{\frac{2}{\tau-3+b}},\left[\frac{1}{2^{\frac{q}{2}-\frac{1}{t^{\prime}}(1-\frac{b}{2})}8D_2a^{\frac{2}{t^{\prime}}(1-\frac{b}{2})}}\right]^{\frac{1}{\frac{q}{2}-\frac{1}{t^{\prime}}(1-\frac{b}{2})-1}} \right\}
\end{equation}
to obtain 
\begin{equation*}
\begin{aligned}
J(v)-J(u)\ge \frac{K}{2}>0.
\end{aligned}
\end{equation*}
This prove \eqref{JA<JB}. At least,  for $u_0\in A$ the assumption $(f_3)$ and \eqref{F<Lps} imply
\begin{equation}\nonumber
\frac{\mu}{p}\|u_0\|^{p}_{L^{p}_{b}}\le  \|F(u_0)\|_{L^{1}_{b}}\le  \zeta \|u_0\|^{\tau+1}_{L^{\tau+1}_{b}}+D_1\|u_0\|^q_{L^{qt^{\prime}}_b}.
\end{equation}
Hence, we must have $D_2=D_1C^{q}(qt^{\prime},b)\to\infty$ when $\mu\to\infty$. So,  \eqref{K-defTM} shows that    $K(a,b,\mu)\to 0$ when $\mu\to \infty$.
\end{proof}
\begin{corollary}\label{coro-cross} Let $K=K(a,b,\mu)$ be given by  \eqref{K-defTM} and $A$  as in Lemma~\ref{L=A<B}. Then $J>0$ on $A$ and 
\begin{equation}\nonumber
J_{**}=\inf\Big\{J(u)\;:\; u\in S(a)\;\;\text{and}\;\; \|\nabla u\|^{2}_{L^2}=\frac{K}{2}\Big\}>0.
\end{equation}
\end{corollary}
\begin{proof}
For any $u\in A$ with  $K$ be given by  \eqref{K-defTM},  we are in a position to repeat the argument in \eqref{F<grads} to obtain
\begin{align*}
J(u) &\ge \frac{1}{2}\|\nabla u\|^{2}_{L^2}-C_1\big(\|\nabla u\|^{2}_{L^2}\big)^{\frac{\tau-1+b}{2}}a^{2-b}-D_2\big(\|\nabla u\|^{2}_{L^2}\big)^{\frac{q}{2}-\frac{1}{t^{\prime}}(1-\frac{b}{2})}a^{\frac{2}{t^{\prime}}(1-\frac{b}{2})}\\
&\ge \|\nabla u\|^{2}_{L^2}\Big[\frac{1}{2}-C_1K^{\frac{\tau-3+b}{2}}a^{2-b}-D_2 K^{\frac{q}{2}-\frac{1}{t^{\prime}}(1-\frac{b}{2})-1}a^{\frac{2}{t^{\prime}}(1-\frac{b}{2})}\Big]\\
& \ge \frac{1}{2}\|\nabla u\|^{2}_{L^2}\Big[1-\frac{1}{2^{\frac{\tau-1+b}{2}}4}-\frac{1}{2^{\frac{q}{2}-\frac{1}{t^{\prime}}(1-\frac{b}{2})}4}\Big]\\
& \ge  \frac{1}{4}\|\nabla u\|^{2}_{L^2}>0.
\end{align*}
Of course, we have $J_{**}\ge K/8$.
\end{proof}
 Now, for fixed $u_0\in S(a)$ we can apply  Lemma~\ref{vicentejTM} and Corollary~\ref{coro-cross} to obtain two numbers $s_1=s_1(u_0,a, b,\mu)<0$ and  $s_2=s_2(u_0,a, b,\mu)>0$ such that the functions $u_{1,\mu}=\mathcal{H}(u_0, s_1)$ and $u_{2,\mu}=\mathcal{H}(u_0, s_2)$ satisfying
 \begin{equation}
 \|\nabla u_{1,\mu}\|^{2}_{L^2}<\frac{K(a,b,\mu)}{2},\;\; \|\nabla u_{2,\mu}\|^{2}_{L^2}>2K(a,b,\mu),\;\; J(u_{1,\mu})>0\;\;\text{and}\;\; J(u_{2,\mu})<0.
 \end{equation}
 Following \cite{jeanjean1}, we fix the  mountain pass level 
 \begin{align*}
 \gamma_{\mu}(a,b)=\inf_{h\in \Gamma}\max_{t\in [0,1]}J(h(t))
 \end{align*}
 where
 \begin{align*}
 \Gamma=\Big\{h\in C([0,1], S(a))\;:\;\|\nabla h(0)\|^{2}_{L^2}<\frac{K(a,b,\mu)}{2}\;\;\text{and}\;\; J(h(1))<0\Big\}.
 \end{align*}
 For each $h\in \Gamma$, Corollary~\ref{coro-cross} yields  $\|\nabla h(1)\|^{2}_{L^2}>K(a,b,\mu)$ and there is $t_0\in [0,1]$ such that  $\|\nabla h(t_0)\|^{2}_{L^2}=\frac{K(a,b,\mu)}{2}$
 and consequently
 \begin{align*}
 \max_{t\in [0,1]}J(h(t))\ge J(h(t_0))\ge J_{**}>0.
 \end{align*}
 So, $ \gamma_{\mu}(a,b)>J_{**}.$
 \begin{lemma} \label{lemmaminx-0}For any $0<b<2$, we have 
 \begin{align*}
 \lim_{\mu\to\infty}\gamma_{\mu}(a,b)=0.
 \end{align*}
 \end{lemma}
 \begin{proof}
 Let us fix the path $h_0(t)=\mathcal{H}(u_0, (1-t)s_1+ts_2)\in\Gamma$. The assumption $(f_3)$ (cf. \eqref{F-low})  implies
 \begin{align*}
 \gamma_{\mu}(a,b)\le  \max_{t\in [0,1]}J(h_0(t))
 \le \max_{r\ge 0}\Big\{\frac{r}{2}\|\nabla u_0\|^{2}_{L^2}-\frac{\mu}{p} r^{\frac{p-2+b}{2}}\|u_0\|^{p}_{L^{p}_{b}}\Big\}.
 \end{align*}
 Since we are assuming $p>4$ we can write
 \begin{align*}
  \gamma_{\mu}(a,b)\le C_2\Big(\frac{1}{\mu}\Big)^{\frac{2}{p-4+b}}\to 0\;\;\mbox{as}\;\; \mu\to \infty
 \end{align*}
 for some $C_2$ that  does not depend on $\mu>0$.
 \end{proof}
Arguing as \cite[Proposition~2.2, Lemma~2.4]{jeanjean1} we can apply the Ekeland's variational principle to functional $\tilde{J}$ in order to  get a Palais-Smale sequence $((v_n,s_n))\subset S(a)\times \mathbb{R}$ which can be used to construct a Palais-Smale sequence $u_n=\mathcal{H}(v_n, s_n)$  with $(u_n)\subset S(a)$ for $J$ at the level $\gamma_{\mu}(a,b)$ satisfying 
 \begin{equation}\label{TM-JPS}
 \begin{aligned}
 J(u_n)=\gamma_{\mu}(a,b)+o_{n}(1)\\
 \end{aligned}
 \end{equation}
 and, for any $\varphi\in H^1(\mathbb{R}^2)$
 \begin{equation}\label{TM-JPSWeak}
 \begin{aligned}
 \int_{\mathbb{R}^2}\nabla u_n\nabla \varphi dx-\int_{\mathbb{R}^2}|x|^{-b}f(u_n)\varphi-\lambda_n \int_{\mathbb{R}^2}u_n\varphi=o_{n}(1)\|\varphi\|.
 \end{aligned}
 \end{equation}
 for some sequence $(\lambda_n)\subset\mathbb{R}$. By \eqref{Pohozaev3} with $G(x,t)=|x|^{-b}F(t)-\frac{\lambda}{2} t^2$, we can verify that any solution $u\in H^{1}(\mathbb{R}^2)$ of \eqref{TM-eq} we must satisfy the identity
\begin{align*}
\int_{\mathbb{R}^2}|\nabla u|^2dx+(2-b)\int_{\mathbb{R}^2}|x|^{-b}F(u)dx-\int_{\mathbb{R}^2}|x|^{-b}f(u)u dx=0.
\end{align*}
 Analogously to \eqref{EQ1**} the limit (cf. \cite[Proposition 2.2]{jeanjean1})
$$
{\partial_s}\tilde{J}(v_n,s_n) \to 0 \quad \mbox{as} \quad n \to \infty,
$$
implies 
\begin{equation}\label{PokoTM}
 P(u_n)=\int_{\mathbb{R}^2}|\nabla u_n|^2dx+(2-b)\int_{\mathbb{R}^2}|x|^{-b}F(u_n)dx-\int_{\mathbb{R}^2}|x|^{-b}f(u_n)u_ndx\to 0
 \end{equation}
 as $n\to\infty$, where $u_n=\mathcal{H}(v_n,s_n).$
 Also, $(u_n)$ is bounded in $H^1(\mathbb{R}^2)$ and from \eqref{TM-JPSWeak} the sequence $(\lambda_n)$ satisfies
 \begin{equation}\label{lambda_nassym}
 \lambda_n=\frac{1}{a^2}\Big\{\int_{\mathbb{R}^2}|\nabla u_n|^2dx-\int_{\mathbb{R}^2}|x|^{-b}f(u_n)u_ndx\Big\}+o_n(1)
 \end{equation}
 
 \begin{lemma}\label{2limMax} Let $\theta>4$ be given by assumption $(f_2)$. Then
 \begin{enumerate}
 \item [$(a)$]
 \begin{equation}\nonumber
 \limsup_{n\to\infty}\int_{\mathbb{R}^2}|x|^{-b}F(u_n)dx\le \frac{2}{\theta-4+b}\gamma_{\mu}(a,b).
 \end{equation}
 \item [$(b)$]
 \begin{equation}\nonumber
  \limsup_{n\to\infty}\int_{\mathbb{R}^2}|\nabla u_n|^{2}dx\le \frac{2(\theta-2+b)}{\theta-4+b}\gamma_{\mu}(a,b).
 \end{equation}
 \end{enumerate}
  In particular, there is $\mu^{*}>0$ such that
  \begin{equation}\nonumber
 \limsup_{n\to\infty}\int_{\mathbb{R}^2}|\nabla u_n|^{2}dx<1-a^2,\;\;\mbox{for all}\;\; \mu\ge \mu^{*}.
  \end{equation}
 \end{lemma}
 \begin{proof}
 From \eqref{TM-JPS} and \eqref{PokoTM},  it follows that
 \begin{equation}\nonumber
 2J(u_n)+P(u_n)=2\gamma_{\mu}(a,b)+o_n(1)
 \end{equation}
 which yields
 \begin{equation}\nonumber
 2\int_{\mathbb{R}^2}|\nabla u_n|^2dx-b\int_{\mathbb{R}^2}|x|^{-b}F(u_n)dx-\int_{\mathbb{R}^2}|x|^{-b}f(u_n)u_n dx=2\gamma_{\mu}(a,b)+o_n(1).
 \end{equation}
 By using  $J(u_n)=\gamma_{\mu}(a,b)+o_{n}(1)$ we also can write
 \begin{equation}\nonumber
 4\gamma_{\mu}(a,b)+(4-b)\int_{\mathbb{R}^2}|x|^{-b}F(u_n)dx -\int_{\mathbb{R}^2}|x|^{-b}f(u_n)u_n dx=2\gamma_{\mu}(a,b)+o_n(1).
 \end{equation}
 Hence, for $\theta>4$ given by $(f_2)$
 \begin{equation}\nonumber
 2\gamma_{\mu}(a,b)+o_n(1)= \int_{\mathbb{R}^2}|x|^{-b}f(u_n)u_n dx-(4-b)\int_{\mathbb{R}^2}|x|^{-b}F(u_n)dx\ge (\theta-4+b)\int_{\mathbb{R}^2}|x|^{-b}F(u_n)dx
 \end{equation}
 which proves  item $(a)$. Now, $J(u_n)=\gamma_{\mu}(a,b)+o_{n}(1)$ ensures
 \begin{equation}\nonumber
 \int_{\mathbb{R}^2}|\nabla u_n|^2dx=2\gamma_{\mu}(a,b)+2\int_{\mathbb{R}^2}|x|^{-b}F(u_n)dx+o_n(1).
 \end{equation}
 Thus, item $(a)$ implies 
 \begin{equation}\nonumber
 \limsup_{n\to\infty}\int_{\mathbb{R}^2}|\nabla u_n|^2dx\le \frac{2(\theta-2+b)}{\theta-4+b}\gamma_{\mu}(a,b).
 \end{equation}
 This proves $(b)$. Further, the existence of $\mu^*>0$ is ensured by Lemma~\ref{lemmaminx-0} and $(b)$ .
 \end{proof}
 \begin{lemma}\label{Lagra} Let $\mu\ge \mu^{*}$, where $\mu^*$ is given by Lemma~\ref{2limMax}.
 Then, the $(\lambda_n)$ is a bounded sequence such that 
 \begin{equation}\nonumber
 \limsup_{n\to\infty}|\lambda_n|\le \frac{4}{a^2}\frac{(\theta-1+b/2)}{\theta-4+b}\gamma_{\mu}(a,b)\;\;\mbox{and}\;\; \limsup_{n\to\infty}\lambda_{n}=-\frac{2-b}{a^2}\liminf_{n\to\infty}\int_{\mathbb{R}^2}|x|^{-b}F(u_n)dx.
 \end{equation}
 \end{lemma}
 \begin{proof}
 Firstly, from \eqref{PokoTM} and Lemma~\ref{2limMax} we obtain
 \begin{equation}\nonumber
 \limsup_{n\to\infty}\int_{\mathbb{R}^2}|x|^{-b}f(u_n)u_ndx\le \frac{2\theta}{\theta-4+b}\gamma_{\mu}(a,b).
 \end{equation}
 From \eqref{lambda_nassym} we get
 \begin{equation}\label{a^2lambda}
 a^2\lambda_n=\int_{\mathbb{R}^2}|\nabla u_n|^2dx-\int_{\mathbb{R}^2}|x|^{-b}f(u_n)u_ndx+o_n(1).
 \end{equation}
 By using Lemma~\ref{2limMax} again,  it follows that $(\lambda_n)$ is a bounded sequence such that 
 \begin{equation}\nonumber
 \limsup_{n\to\infty}|\lambda_n|\le \frac{4}{a^2}\frac{(\theta-1+b/2)}{\theta-4+b}\gamma_{\mu}(a,b).
 \end{equation}
 In addition, by combining \eqref{PokoTM} and  \eqref{a^2lambda} we have 
 $$
  \limsup_{n\to\infty}\lambda_{n}=-\frac{2-b}{a^2}\liminf_{n\to\infty}\int_{\mathbb{R}^2}|x|^{-b}F(u_n)dx.
 $$
 \end{proof}
Now, since $(u_n)$ is bounded in $H^{1}(\mathbb{R}^2)$, we can assume $u_n \rightharpoonup u$ in $H^{1}(\mathbb{R}^2)$. The next result shows that  $u\not\equiv 0$, at least for $\mu\ge \mu^*$.

\begin{lemma}\label{weak=no0} Assume that the  Palais-Smale sequence $(u_n)$ is such that  $u_n \rightharpoonup u$ in $H^{1}(\mathbb{R}^2)$. Let $\mu^*$ be given by Lemma~\ref{2limMax}. Then, for any $\mu\ge \mu^*$ we must have  $u\not\equiv0$.
\end{lemma}
\begin{proof}
From Lemma~\ref{lemma-conv} and Lemma~\ref{2limMax} it follows that
 $$F(u_n)\to F(u) \;\;\mbox{and} \;\; f(u_n)u_n\to f(u)u\;\;\mbox{in} \;\; L^{1}_{b}(\mathbb{R}^2).$$
 By contradiction, suppose that $u\equiv 0$. Thus,
 $$\lim_{n\to\infty}\int_{\mathbb{R}^2}|x|^{-b}F(u_n)dx=\lim_{n\to \infty}\int_{\mathbb{R}^2}|x|^{-b}f(u_n)u_n dx=0$$
 and Lemma~\ref{Lagra} ensures
 $$
   \limsup_{n\to\infty}\lambda_{n}\le 0.
 $$
 By using Lemma~\ref{lemma-conv} again and the identity \eqref{a^2lambda} we conclude  that
 $$
 \lambda_n a^2=  \|\nabla u_n\|^2_{L^2}+o_n(1)
 $$
 and consequently
 $$0\le \liminf_{n\to\infty}\|\nabla u_n\|^2_{L^2}\le \limsup_{n\to\infty}\|\nabla u_n\|^2_{L^2}=\limsup_{n\to\infty}\lambda_n a^2\le 0$$
 which leads to  $\|\nabla u_n\|^2_{L^2}\to 0$ and contradicts $\gamma_{\mu}(a,b)>0$. Thus, we must have $u\not\equiv0$.
\end{proof}
\section{Proof of Theorem~\ref{thm-TM}}
Let $u$ be the weak limit of the  Palais-Smale sequence $(u_n)\subset S(a)$ such as in  \eqref{TM-JPS}, \eqref{TM-JPSWeak}, \eqref{PokoTM} and  \eqref{lambda_nassym}. From Lemma~\ref{weak=no0}, we must have $u\not\equiv 0$. For $\mu>\mu^*$,  Lemma~\ref{lemma-conv},   Lemma~\ref{2limMax}  and  Lemma~\ref{Lagra} imply
\begin{equation}\nonumber
 \limsup_{n\to\infty}\lambda_{n}=-\frac{2-b}{a^2}\liminf_{n\to\infty}\int_{\mathbb{R}^2}|x|^{-b}F(u_n)dx=-\frac{2-b}{a^2}\liminf_{n\to\infty}\int_{\mathbb{R}^2}|x|^{-b}F(u)dx<0.
 \end{equation}
Thus, up to a subsequence, we can assume
$$
\lim_{n\to\infty}\lambda_n=\tilde{\lambda}_{a}<0.
$$
From \eqref{TM-JPSWeak}, $u$ satisfies the equation 
 \begin{equation}\nonumber
 -\Delta u-|x|^{-b}f(u)=\tilde{\lambda}_{a} u \;\;\mbox{in}\;\; \mathbb{R}^2
 \end{equation}
from which
 \begin{equation}\label{fuuuu}
 \|\nabla u\|^{2}_{L^{2}}-\tilde{\lambda}_a \|u\|^2_{L^2} = \int_{\mathbb{R}^2} |x|^{-b}f(u)udx.
 \end{equation}
But, 
 \begin{equation}\nonumber
 \|\nabla u_n\|^{2}_{L^{2}}-\lambda_n \|u_n\|^2_{L^2} = \int_{\mathbb{R}^2} |x|^{-b}f(u_n)u_ndx+o_n(1)
 \end{equation}
 which also yields
 \begin{equation}\nonumber
 \|\nabla u_n\|^{2}_{L^{2}}-\tilde{\lambda}_{a} \|u_n\|^2_{L^2} = \int_{\mathbb{R}^2} |x|^{-b}f(u_n)u_ndx+o_n(1).
 \end{equation}
 Thus, from Lemma~\ref{lemma-conv} and \eqref{fuuuu} we also have 
  \begin{equation}\nonumber
 \lim_{n\to\infty}(\|\nabla u_n\|^{2}_{L^{2}}-\tilde{\lambda}_{a} \|u_n\|^2_{L^2} )=\|\nabla u\|^{2}_{L^{2}}-\tilde{\lambda}_a \|u\|^2_{L^2}. 
 \end{equation}
 Since $\tilde{\lambda}_{a}<0$, the expression
 $$\|u\|_{\lambda_a}=(\|\nabla u\|^{2}_{L^{2}}-\tilde{\lambda}_a \|u\|^2_{L^2})^{{1}/{2}}$$ defines a norm on $H^1(\mathbb{R}^2)$ which is equivalent to the classical $\|u\|$. So, the above limit implies that $u_n\to u$ strongly in $H^1(\mathbb{R}^2)$.
 Hence,  $u\in S(a)$ and it is a non-trivial  solution for the equation in  \eqref{problemPP} with $\lambda=\lambda_a=-\tilde{\lambda}_{a}>0$.

\end{document}